\documentclass[12pt]{article}
\usepackage{graphicx,amsmath,amsfonts,amssymb,amscd,amsthm, mathrsfs}
\newtheoremstyle{kai}
{3pt}{3pt}{}{}{\bfseries}{.}{.5em}{}
\makeatletter
\def\EquationsBySection{\def\theequation
{\thesection.\arabic{equation}}%
\@addtoreset{equation}{section}}

\newcommand\old[1]{}
\newcommand{\pend}{\hfill \thicklines \framebox(6.6,6.6)[l]{}}
\renewenvironment{proof}{\noindent {\it  Proof.} \rm}{\pend}
\newtheorem{theorem}{Theorem}[section]
\newtheorem{lemma}{Lemma}[section]
\newtheorem{corollary}{Corollary}[section]
\newtheorem{proposition}{Proposition}[section]

\EquationsBySection
\makeatother

\begin{document}
\pagestyle{plain}

\title
{\bf On Regularity of Stochastic Convolutions of \\Functional Linear Differential\\ Equations with Memory}
\author{
\small{}
Kai Liu$^{a,\, b}$
\\
\\
\small{$^a$College of Mathematical Sciences,}\\
\small{Tianjin Normal University,}\\
\small{300387, Tianjin, P. R. China.}\\
\\
\small{$^b$Department of Mathematical Sciences,}\\
\small{School of Physical Sciences,}\\
\small{The University of Liverpool,}\\
\small{Peach Street, Liverpool, L69 7ZL, U.K.}\\
\small{E-mail: k.liu@liverpool.ac.uk}\\}

\date{}
\maketitle
\noindent {\bf Abstract:}  In this work, we consider the regularity property of stochastic convolutions for a class of  abstract linear stochastic retarded functional differential equations with unbounded operator coefficients. We  first establish some useful estimates on fundamental solutions which are time delay versions of those on $C_0$-semigroups. To this end, we develop a scheme of constructing the resolvent operators for the integrodifferential equations of Volterra type since the equation under investigation is of this type in each subinterval describing the segment of its solution. Then we apply these estimates to  stochastic convolutions of our equations to obtain the desired regularity property.

\vskip 30pt
\noindent{\bf 2000 Mathematics Subject Classification(s):} 60H15, 60G15, 60H05.

\newpage

\section{Introduction}

We begin with an example of stochastic delay heat equations without exterior energy source to motivate our work.
Let $h_i:\, {\mathbb R}\to {\mathbb R}$, $i=0,\,1$, be two monotonically differentiable functions which satisfy the following conditions (see, e.g., Coleman and Gurtin \cite{cbdgme196705} and Nunziato \cite{njw1971})
\[
\begin{split}
x \cdot h_i(x) & \ge \gamma_{i, 1}|x|^p +\alpha_{i, 1},\hskip 15pt \forall\,x\in {\mathbb R},\\
|h_i(x)| &\le \gamma_{i, 2}|x|^{p-1} + \alpha_{i, 2},\hskip 15pt \forall\, x\in {\mathbb R},
\end{split}
\]
where $p\ge 2$, $\gamma_{i, j}>0$ and $\alpha_{i, j}\in {\mathbb R}$ for $i=0,\,1$ and $j=1,\,2$.
The non-Fourier heat conduction model with delay in the conductor $(0, \pi)\subset {\mathbb R}$ starts from the following constitutive equation
\begin{equation}
\label{01/10/17(1)}
q(t, x) = -h_0(\partial y(t, x)/\partial x)- h_1(y(t-r, x)),\hskip 15pt t\ge 0,\hskip 15pt x\in (0, \pi),
\end{equation}
 and the energy conservative equation without exterior energy sources
\begin{equation}
\label{01/10/17(200)}
d_tp(t, x) + \frac{\partial q(t, x)}{\partial x}dt=0,\hskip 15pt t\ge 0,\hskip 15pt x\in (0, \pi),
\end{equation}
where $y$ denotes the temperature, $q$ is the heat flux and $p$ is the internal energy which can be taken, in most situations, as the form: $p(t, x)=\kappa y(t, x)$, $\kappa>0$.
In practice, the assumption of zero exterior energy source is artificial, and a more realistic model is that the null exterior energy source is perturbed by a noise process, for example, a Gaussian white noise $b(x)\dot w(t, x)$, $b\in L^2(0, \pi)$. In other words, we replace (\ref{01/10/17(200)}) by the equation
\begin{equation}
\label{01/10/17(3)}
d_tp(t, x) + \frac{\partial q(t, x)}{\partial x}dt= b(x)dw(t, x),\hskip 15pt t\ge 0,\hskip 15pt x\in (0, \pi).
\end{equation}
 Then, by substituting (\ref{01/10/17(3)}) into (\ref{01/10/17(1)}), we obtain, for simplicity, letting $\kappa=1$, the following equation
 \begin{equation}
\label{01/10/17(10)}
 \begin{cases}
 dy(t, x) = \displaystyle\frac{\partial h_0(\partial y(t, x)/\partial x)}{\partial x}dt + \frac{\partial h_1(y(t-r, x))}{\partial x}dt+ b(x)dw(t, x),\hskip 10pt t\ge 0,\\
 y(0, \cdot)=\phi_0(\cdot)\in L^2(0, \pi),\,\,y(\theta, \cdot)=\phi_1(\theta, \cdot)\in W^{1, p}_0(0, \pi),\,\, \theta\in [-r, 0],\\
 y(t, 0)= y(t, \pi) =0,\,\,\,\,\,t\in (0, \infty).
 \end{cases}
 \end{equation}
Let $\Delta =d^2/dx^2$, $H=L^2(0, \pi)$, $V= H^1_0(0, \pi)$, $U=W^{1, p}_0(0, \pi)$ and $Bu=bu$, $u\in H$, $W(t)=w(t, \cdot)$, $g_0(u)=\displaystyle\frac{d h_0(du(x)/dx)}{dx}$, $g_1(u)=\displaystyle\frac{d h_1(u(x))}{dx}$ for any $u\in W^{1, p}_0(0, \pi)$. We have thus a stochastic differential equation with delay in $H$,
\begin{equation}
\label{01/10/17(4)}
\begin{cases}
dy(t) = g_0(y(t))dt + g_1(y(t-r))dt + BdW(t),\,\,\,\,t\ge 0,\\
y(0)=\phi_0,\,\,y_0=\phi_1,
\end{cases}
\end{equation}
where $U\subset H\subset U^*$ and $g_i$, $i=0,\,1$, is a continuous monotone operator from $U$ to $U^*$ such that
\[
\begin{split}
\langle u, g_i(u)\rangle_{U, U^*}&\ge \gamma_{i, 1}\|u\|^p_U +\alpha_{i, 1},\hskip 15pt \forall\,u\in U,\\
\|g_i(u)\|_{U^*}&\le \gamma_{i, 2}\|u\|^{p-1}_U + \alpha_{i, 2},\hskip 15pt \forall\, u\in U,
\end{split}
\]
where $p\ge 2$, $\gamma_{i, j}>0$ and $\alpha_{i, j}\in {\mathbb R}$ for $i=0,\,1$ and $j=,\,2$. In particular, if $h_0(x)=x$, $h_1(x)=\gamma x$, $\gamma>0$, $p=2$, then $g_0(u)=\Delta u$, $g_1(u)=\gamma(-\Delta)^{1/2} u$, $V=U$ and the equation (\ref{01/10/17(4)}) reduces to
\begin{equation}
\label{01/10/17(5)}
\begin{cases}
dy(t) = \Delta y(t)dt + \gamma(-\Delta)^{1/2} y(t-r)dt + BdW(t),\,\,\,\,t\ge 0,\\
y(0)=\phi_0,\,\,y_0=\phi_1.
\end{cases}
\end{equation}
The aim of this work is to investigate the regularity property of such stochastic systems  as (\ref{01/10/17(5)}).

The organization of this  work is as follows. In Section 2, we first introduce the deterministic linear retarded functional differential equation associated in our formulation of stochastic systems. We review the useful variation of constants formula for the equation under consideration by means of its fundamental solution. Also, we state some estimates about fundamental solutions which will play an important role in the subsequent investigation. By employing the main results, we establish in Section 3 the desired regularity property of stochastic convolutions. In Sections 4 and 5, we present the detailed proofs of the main theorem, i.e., Theorem 2.1, by following J. Pr\"uss's method of constructing the resolvent operators for the integrodifferential equations of Volterra type.

\section{Fundamental Solution}

We are concerned with the following linear retarded functional differential equation in a Banach space $X$,
\begin{equation}
\label{13/06/14(1)}
\begin{cases}
{d}y(t) = Ay(t)dt + A_1y(t-r)dt + \displaystyle\int^0_{-r} a(\theta)A_2 y(t+\theta)d\theta +f(t), \,\,\,\,t\ge 0,\\
y(0) = \phi_0, \,\,y(\theta) =\phi_1(\theta),\,\,\theta\in [-r, 0],\,\,\,\phi = (\phi_0, \phi_1),
\end{cases}
\end{equation}
where $r>0$ is some constant incurring the system delay, $a(\cdot)\in L^2([-r, 0], {\mathbb R})$ and $\phi = (\phi_0, \phi_1)$ is an appropriate initial datum.
Here $A: {\mathscr D}(A)\subset X\to X$ is the infinitesimal generator of an analytic semigroup $e^{tA}$, $t\ge 0$, and $A_1$, $A_2$ are two closed linear operators with domains  ${\mathscr D}(A_i)\supset {\mathscr D}(A)$, $i=1,\,2$, and $f$ is a continuous function with values in $X$.  For simplicity, we assume in this work that the $C_0$-semigroup $e^{tA}$ is negative type, i.e., there exist constants $M\ge 1$ and $\mu>0$ such that
\begin{equation}
\label{10/06/18(1)}
\|e^{tA}\|\le Me^{-\mu t},\hskip 15pt  \|Ae^{tA}\|\le M/t\hskip 15pt \hbox{for all}\hskip 15pt t> 0,
\end{equation}
and for $\gamma\in (0, 1)$, there exists a constant $M_\gamma>0$ such that
\begin{equation}
\label{10/06/18(2)}
  \|(-A)^\gamma e^{tA}\|\le M_\gamma/t^\gamma\hskip 15pt \hbox{for all}\hskip 15pt t> 0,
\end{equation}
where $(-A)^\gamma$ is the standard fractional power of operator $A$.

Equations of the type (\ref{13/06/14(1)}) were investigated by Di Blasio, Kunisch and Sinestrari \cite{Gdbkkes85(2)}, Sinestrari \cite{es83, es84} and the fundamental solution to (\ref{13/06/14(1)}) was introduced by Jeong, Nakagiri and Tanabe \cite{jjsnht93}. In particular, it is known that the fundamental solution $G(\cdot): {\mathbb R}\to {\mathscr L}(X)$ to (\ref{13/06/14(1)}) is an operator-valued function which is strongly continuous in $X$ and satisfies
\begin{equation}
\label{27/05/18(1)}
\begin{split}
\displaystyle\frac{d}{dt}G(t) &= AG(t) + A_1G(t-r) + \int^0_{-r} a(\theta)A_2G(t+\theta)d\theta,\\
&\hskip 20pt G(0)=I,\hskip 15pt G(t) ={\rm O},\,\,\,\,\,\,\,t\in (-\infty, 0),
\end{split}
\end{equation}
where  ${\rm O}$ is the null operator in $X$.
According to the well-known Duhamel's principle, the problem (\ref{27/05/18(1)}) is transformed to the integral equation
\begin{equation}
\label{25/05/06(1578)}
G(t) = \begin{cases}
e^{tA} + \displaystyle\int^t_0 e^{(t-s)A} A_1G(s-r)ds  + \displaystyle\int^t_0 \displaystyle\int^0_{-r} a(\theta) e^{(t-s)A} A_2 G(s+\theta)d\theta ds,\hskip 15pt &t\ge 0,\\
{\rm O},\hskip 15pt &t< 0.
\end{cases}
\end{equation}
The fundamental solution $G$ enables us to solve the initial value problem for the equation (\ref{13/06/14(1)}). In fact, it may be shown that under some reasonable conditions on $f$ and initial datum $\phi=(\phi_0, \phi_1)$, the unique mild solution $y$ of (\ref{13/06/14(1)}) is represented as
\begin{equation}
\label{27/05/18(2)}
y(t) = G(t)\phi_0 + \int^0_{-r} U_t(\theta)\phi_1(\theta)d\theta + \int^t_0 G(t-s)f(s)ds,\hskip 15pt t\ge 0,
\end{equation}
where
\[
U_t(\theta) = G(t-\theta-r)A_1 + \int^\theta_{-r} G(t-s+\tau)a(\tau)A_2d\tau,\hskip 20pt \theta\in [-r, 0],\]
with the initial condition $y(0)=\phi_0$ and $y(\theta)=\phi_1(\theta)$, $\theta\in [-r, 0)$. This is a time delay version of the usual variation of constants formula without memory.

In order to apply (\ref{27/05/18(2)}) to such equations as (\ref{01/10/17(5)}) to consider their regularity property of solutions, we need establish some inequalities in association with $G(\cdot)$, which are the main results of this work. To this end, we shall formulate the following condition:
\begin{enumerate}
\item[(H)]
$a(\cdot)\in L^\infty([-r, 0], {\mathbb R})$ and
\begin{equation}
\label{14/09/18(1)}
 {\mathscr D}((-A)^\mu)\subset {\mathscr D}(A_1),\hskip 20pt  {\mathscr D}((-A)^\nu)\subset {\mathscr D}(A_2)
\end{equation}
for some $0< \mu, \,\nu< 1$.
\end{enumerate}

\begin{theorem}
\label{21/05/18(1)}
Assume that condition (H) holds. Then
\begin{enumerate}
\item[(a)] for any $\gamma\in [\nu, 1)$, it is true that
\begin{align}
&\displaystyle\| (-A)^{\gamma}G(t)\|\le \frac{C_{n, \gamma}}{(t-nr)^{\gamma}}\hskip 15pt \hbox{for all}\hskip 15pt t\in (nr, (n+1)r],\label{4567}\\
&\Big\|\displaystyle\int^t_{s}(-A)^\gamma G(u)du\Big\|\le C_{n, \gamma}\hskip 15pt \hbox{for all}\hskip 15pt nr\le s< t\le (n+1)r,\label{4569}
\end{align}
where $C_{n, \gamma}>0$, $n\in {\mathbb N} := \{0,\,1,\,\cdots\}$, are constants depending on $n$ and $\gamma$.
\item[(b)] for any $\gamma\in [\nu, 1)$ and $0<\beta< 1-\gamma$, it is true that
\begin{align}
&\|(-A)^\gamma (G(t)-G(s))\|\le C_{n, \beta, \gamma}\displaystyle\frac{(t-s)^{\beta}}{(s-nr)^{\beta+\gamma}}\hskip 10pt \hbox{for all}\hskip 10pt nr< s< t<  (n+1)r,\label{14/06/18(80)}
\end{align}
where $C_{n, \beta, \gamma}>0$, $n\in {\mathbb N}$, are constants depending on $n,$ $\beta$ and $\gamma$.
\end{enumerate}
\end{theorem}

\section{Stochastic Convolution}

Let $\{\Omega, {\mathscr F}, {\mathbb P}\}$ be a
 probability space equipped with some filtration $\{\mathscr{F}_{t}\}_{t\geq 0}$. Let $K$ be a separable Hilbert space and  $\{W_Q(t),\,t\ge 0\}$ denote a $Q$-Wiener process with respect to $\{\mathscr{F}_{t}\}_{t\geq 0}$ in $K$, defined on  $\{\Omega, {\mathscr F}, {\mathbb P}\}$, with covariance operator $Q$, i.e.,
\[
\mathbb{E}\langle W_Q(t), x\rangle_K\langle W_Q(s), y\rangle_K = (t\wedge s)\langle Qx, y\rangle_K\,\,\,\,\hbox{for all}\,\,\,\,\, x,\,\,y\in K,\]
where $Q$ is a positive, self-adjoint and trace class operator on $K$. We frequently call $W_Q(t)$, $t\ge 0$, a $K$-valued $Q$-Wiener process with respect to $\{\mathscr{F}_{t}\}_{t\geq 0}$ if the trace $Tr(Q)<\infty$.
We introduce a subspace $K_Q={\mathscr R}(Q^{1/2})\subset K$, the range of $Q^{1/2}$, which is a Hilbert space endowed with the inner product
\[
\langle u, v\rangle_{K_Q} =\langle Q^{-1/2}u, Q^{-1/2}v\rangle_K\hskip 10pt\hbox{for any}\hskip 10pt  u,\,\,v\in K_Q.\]
Let $H$ be a separable Hilbert space and ${\mathscr L}_2(K_Q, H)$ denote the space of all Hilbert-Schmidt operators from $K_Q$ into $H$. Then ${\mathscr L}_2(K_Q, H)$ turns out to be a separable Hilbert space, equipped with the norm
\[
\|\Psi\|^2_{{\mathscr L}_2(K_Q, H)} =Tr [\Psi Q^{1/2}(\Psi Q^{1/2})^*]\hskip 15pt \hbox{for any}\,\,\,\,\Psi\in {\mathscr L}_2(K_Q, H).\]
For arbitrarily given $T\ge 0$, let
$J(t,\omega)$, $t\in[0,T]$, be an ${\mathscr L}_2(K_Q, H)$-valued process, and we define the following norm for arbitrary $t\in[0,T]$,
\begin{equation}
\label{11/02/09(10)}
|J|_t :=\biggl\{\mathbb{E}\int^t_0 Tr\Big[J(s,\omega)Q^{1/2}(J(s,\omega)Q^{1/2})^*\Big]ds\biggr\}^{\frac{1}{2}}.
\end{equation}
In particular, we denote by ${\cal U}^2\big([0,T]; \,{\mathscr L}_2(K_Q, H)\big)$ the space of all ${\mathscr L}_2(K_Q, H)$-valued measurable processes $J$, adapted to the filtration $\{{\mathscr F}_t\}_{t\le T}$, satisfying $|J|_T <
\infty$.

Suppose that $W(\cdot)$ is a $Q$-Wiener process in $K$ such that
$Qe_j= \lambda_j e_j$, $\lambda_j>0$, $j\ge 1$,
where $\{e_j\}$ is a complete orthonormal basis in $K$, then it is immediate that
\[
W(t) = \sum^\infty_{j=1}\sqrt{\lambda_j}w_j(t)e_j,\hskip 15pt t\ge 0,\]
where $\{w_j(\cdot)\}$ is a group of independent real Wiener processes.
The stochastic integral $\displaystyle\int^t_0 J(s)dW(s) \in H$, $t\ge 0$, may be defined
 for all $J\in {\cal U}^2([0,T]\times\Omega;\, {\mathscr L}_2(K_Q, H))$ by
\begin{equation}
\label{08/06/11(1089)}
\int^t_0 J(s)dW(s) = L^2 - \lim_{n\rightarrow \infty}
\sum^n_{j=1} \int^t_0 \sqrt{\lambda_j}J(s)e_j dw_j(s),\hskip 15pt
t\in [0, T].
\end{equation}
The reader is referred to Da Prato and Zabczyk  \cite{gdajz92} for more details on this topic.

We are concerned about the following linear stochastic retarded functional differential equation on $H$,
\begin{equation}
\label{13/06/14(1098)}
\begin{cases}
{d}y(t) = Ay(t)dt + A_1y(t-r)dt + \displaystyle\int^0_{-r} a(\theta)A_2 y(t+\theta)d\theta + BdW(t), \,\,\,\,t\ge 0,\\
y(0) = \phi_0, \,\,y(\theta) =\phi_1(\theta),\,\,\theta\in [-r, 0],\,\,\,\phi = (\phi_0, \phi_1),
\end{cases}
\end{equation}
where $A$, $A_1$, $A_2$ are given as in Section 2, $B\in {\mathscr L}_2(K_Q, H)$ and $\phi=(\phi_0, \phi_1)$ is an appropriate initial datum.
It is well known that the unique mild solution $y$ of (\ref{13/06/14(1098)}) is represented as
\begin{equation}
\label{27/05/18(20953)}
y(t) = G(t)\phi_0 + \int^0_{-r} U_t(\theta)\phi_1(\theta)d\theta + \int^t_0 G(t-s)BdW(s),\hskip 15pt t\ge 0,
\end{equation}
where
\[
U_t(\theta) = G(t-\theta-r)A_1 + \int^\theta_{-r} G(t-s+\tau)a(\tau)A_2d\tau,\hskip 20pt \theta\in [-r, 0],\]
with the initial condition $y(0)=\phi_0$ and $y(\theta)=\phi_1(\theta)$, $\theta\in [-r, 0)$. In particular, if $\phi=(0, 0)$, then the unique mild solution (\ref{27/05/18(20953)}) is the so-called {\it stochastic convolution\/} process
\begin{equation}
\label{27/05/18(2095367)}
y(t) := W_G(t) = \int^t_0 G(t-s)BdW(s),\hskip 15pt t\ge 0.
\end{equation}
For any $T>0$, let $C^\beta([0, T]; H)$ denote the usual Banach space of all H\"older continuous functions on $H$ with order $\beta\in (0, 1)$. The following lemma is referred to Tanabe \cite{ht88(1)}.

\begin{lemma}
\label{06/06/18(1)}
Suppose that $a(\cdot)$ in (\ref{13/06/14(1098)}) is H\"older continuous with order $\rho\in (0, 1)$, then operator  $G(t)$ is strongly continuous in $H$ on each $[nr, (n+1)r]$, $n=0,\,1,\,\cdots$. Moreover, the following estimates hold:
\begin{equation}
\|G(t)-G(s)\|\le C_{n, \kappa}\Big(\displaystyle\frac{t-s}{s-nr}\Big)^{\kappa}\hskip 15pt \hbox{for all}\hskip 15pt nr< s< t< (n+1)r,
\end{equation}
where $\kappa\in (0, \rho)$ and $C_{n, \kappa}>0$ are some constants depending on $n,$ $\kappa$.
\end{lemma}

\begin{theorem}
Suppose that $a(\cdot)$ in (\ref{13/06/14(1098)}) is H\"older continuous with order $\rho\in (0, 1)$.
Let $T>0$. Assume that $Tr(Q)<\infty$ and $B\in {\mathscr L}(K, H)$, the space of all bounded, linear operators from $K$ into $H$, then  the trajectories of $W_G$ are in $C^\beta([0, T]; H)$ where
\[
\beta< \begin{cases}
1/2\hskip 20pt  \hbox{if}\hskip 15pt  \rho>1/2,\\
\rho\hskip 31pt \hbox{if}\hskip 15pt \rho\le 1/2.
\end{cases}
\]
\end{theorem}
\begin{proof} It suffices to show this theorem for any $0\le s<t\le T$ with $t-s<r.$ To this end, it is easy to have that
\begin{equation}
\label{15/06/14(1)}
\begin{split}
{\mathbb E}&\|W_G(t) - W_G(s)\|^2_H\\
 &= \sum^\infty_{j=1}\lambda_j\int^t_s \|G(t-u)Be_j\|^2_H du + \sum^\infty_{j=1} \lambda_j\int^s_0 \|(G(t-u)-G(s-u))Be_j\|^2_H du\\
&=: I_1 + I_2.
\end{split}
\end{equation}
Since $G(\cdot)$ is strongly continuous on ${\mathbb R}$, it is easy to see, by the well-known Principle of Uniform Boundedness, that $G(\cdot)$ is norm bounded on $[0, T]$, and there exists a real number $C(T)>0$ such that
\begin{equation}
\label{15/06/14(2)}
\begin{split}
I_1 &= \sum^\infty_{j=1} \lambda_j \int^t_s \|G(t-u)Be_j\|^2_H du\\
&\le C(T)\|B\|^2\sum^\infty_{j=1}\lambda_j (t-s) = C(T)\|B\|^2Tr(Q)(t-s).
\end{split}
\end{equation}
 On the other hand, suppose that $s\in (Nr, (N+1)r)$ for some $N\in {\mathbb N}$ and $t-s<r$. Then,  for the item $I_2$ we have
\[
\begin{split}
I_2 &= \sum^\infty_{j=1} \lambda_j\int^s_0 \|[G(t-u)-G(s-u)]Be_j\|^2_H du\\
&\le \|B\|^2Tr(Q) \sum^{N}_{k=0}\int^{(k+1)r}_{kr}\|G(t-s+u) -G(u)\|^2 du\\
&\le \|B\|^2Tr(Q) \sum^N_{k=0}\int_{kr}^{(k+1)r-(t-s)}\|G(t-s+u)-G(u)\|^2du\\
&\,\,\,\,\,\, + \|B\|^2Tr(Q) \sum^N_{k=0}\int^{(k+1)r}_{(k+1)r-(t-s)}\|G(t-s+u)-G(u)\|^2du.
\end{split}
\]
By using Lemma  \ref{06/06/18(1)} for those values
\[
\beta< \begin{cases}
1/2\hskip 20pt  \hbox{if}\hskip 15pt  \rho>1/2,\\
\rho\hskip 31pt \hbox{if}\hskip 15pt \rho\le 1/2,
\end{cases}
\]
  one can further obtain
\begin{equation}
\label{15/06/14(3)}
\begin{split}
I_2 & \le \|B\|^2Tr(Q)\sum^N_{k=0}\int_{kr}^{(k+1)r-(t-s)}C_{k, \beta}(t-s)^{2\beta}(u-kr)^{-2\beta}du\\
&\,\,\,\,\,\, +
\|B\|^2Tr(Q)\sum^N_{k=0}\int^{(k+1)r}_{(k+1)r-(t-s)}C_{k, \beta}(t-s)^{2\beta}\Big[u - \big((k+1)r - (t-s)\big)\Big]^{-2\beta}du\\
&\le \|B\|^2Tr(Q)\frac{(t-s)^{2\beta}}{1-2\beta}\Big[ \sum^N_{k=1}C_{k, \beta}(r-(t-s))^{1-2\beta} + \sum^N_{k=1} C'_{k, \beta}(t-s)^{1-2\beta}\Big]\\
&\le \|B\|^2Tr(Q)\Big[\frac{(t-s)^{2\beta}r^{1-2\beta}}{1-2\beta}\sum^N_{k=1}C_{k, \beta} + \frac{t-s}{1-2\beta}\sum^N_{k=1}C'_{k, \beta}\Big]\\
&\le C(T, \beta)\big[(t-s)^{2\beta} + (t-s)\big]\hskip 15pt \hbox{for some}\hskip 15pt C(T, \beta)>0,
\end{split}
\end{equation}
where $C_{k, \beta}>0$, $C'_{k, \beta}>0$ are two numbers depending on $k$ and $\beta$.
Thus, by substituting  (\ref{15/06/14(2)}) and   (\ref{15/06/14(3)}) into (\ref{15/06/14(1)}),   we  immediately obtain
\begin{equation}
\label{20/09/18(1)}
\begin{split}
{\mathbb E}\|W_G(t) - W_G(s)\|^2_X &\le [C(T)\|B\|^2Tr(Q) + C(T, \beta)](t-s) +  C(T, \beta)(t-s)^{2\beta}\\
&\le M(T, \beta)(t-s)^{2\beta}
\end{split}
\end{equation}
where
\[
M(T, \beta) = [C(T)\|B\|^2Tr(Q) +C(T, \beta)]r^{1-2\beta} +C(T, \beta)>0.\]
 From (\ref{20/09/18(1)}), it follows further that for any integer $m\ge 1$,
\[
{\mathbb E}\|W_G(t)-W_G(s)\|^{2m}\le C_m(t-s)^{2m\beta},\hskip 15pt \forall\,0< s<t< T,\,\,t-s<r.\]
So, by the well-known Kolmogorov test, $W_G$ is $\alpha_m$-H\"older continuous with
\[
\alpha_m = \frac{2m\beta-1}{2m}.\]
Since $m$ is arbitrary, the trajectories of $W_G$ are in $C^\beta([0, T]; H)$. The proof is  complete.
\end{proof}

\begin{theorem} Suppose that condition (H) holds. Let $T>0$ and $Tr(Q)<\infty$. For $\nu\le \gamma<1$ and $0<\beta< \displaystyle\frac{1}{2} -\gamma$, the trajectories of $W_G$ are in $C^{\beta}([0, T], {\mathscr D}((-A)^\gamma))$.
\end{theorem}

\begin{proof}
Once again, we intend to use the Kolmogorov test. Let $0\le s< t\le T$ with $s,\,t\in (Nr, (N+1)r)$ for some $N>0$. Then, by definition, it follows easily that
\begin{equation}
\label{06/06/18(10)}
\begin{split}
&{\mathbb E}\|(-A)^\gamma W_G(t) - (-A)^\gamma W_G(s)\|^2_H\\ &= \sum^\infty_{k=1} \lambda_k \int^t_s \|(-A)^\gamma G(t-u)e_k\|^2_Hdu + \sum^\infty_{k=1} \lambda_k \int^s_0 \|(-A)^\gamma [G(t-u)-G(s-u)]e_k\|^2_Hdu\\
&=: I_1 + I_2.
\end{split}
\end{equation}
Now we estimate $I_1$ and $I_2$ separately. First, by using Theorem \ref{21/05/18(1)}
we have for $s,\,t\in (Nr, (N+1)r)$ with $s<t$ that
\begin{equation}
\label{06/06/18(11)}
\begin{split}
I_1 &\le Tr(Q) \int^{t-s}_0 \|(-A)^\gamma G(v)\|^2_Hdv\\
&\le Tr(Q) \int^{t-s}_0 \frac{C^2_{0, \gamma}}{v^{2\gamma}}dv\\
& = \frac{Tr(Q)C^2_{0, \gamma}}{1-2\gamma}(t-s)^{1-2\gamma},\hskip 15pt C_{0, \gamma}>0.
\end{split}
\end{equation}
Since $0< \beta <\displaystyle\frac{1}{2}-\gamma<1-\gamma$, it follows that $1-2\beta-2\gamma>0$, and we thus employ Theorem \ref{21/05/18(1)} to obtain
\begin{equation}
\label{06/06/18(12)}
\begin{split}
I_2 &= \sum_{k=1}^\infty \lambda_k \int^s_0 \|(-A)^\gamma \big[ G(t-s+v) - G(v)\big]e_k\|^2_H dv\\
&= \sum_{k=1}^\infty\lambda_k \sum^{N-1}_{j=0} \int^{(j+1)r}_{jr} \|(-A)^\gamma [G(t-s+v) - G(v)]e_k\|^2_Hdv\\
&\,\,\,\,\,\, +  \sum_{k=1}^\infty\lambda_k\int^{s}_{Nr} \|(-A)^\gamma[G(t-s+v) - G(v)]e_k\|^2_Hdv\\
&\le  Tr(Q)
\Big(\sum_{j=0}^{N-1} \int^{(j+1)r}_{jr} \frac{C^2_{j, \beta, \gamma}(t-s)^{2\beta}}{(v-jr)^{2\beta+2\gamma}}dv +
\int^{s}_{Nr} \frac{C^2_{N, \beta, \gamma}(t-s)^{2\beta}}{(v-Nr)^{2\beta+2\gamma}} dv\Big)\\
&\le  \frac{Tr(Q)}{1-2\beta-2\gamma}\cdot r^{1-2\beta-2\gamma}\cdot \Big(\sum^{[T]+1}_{j=0} C^2_{j, \beta, \gamma}\Big)(t-s)^{2\beta}, \hskip 15pt C_{j, \beta, \gamma}>0,
\end{split}
\end{equation}
where $[T]$ denotes the biggest integer less than or equal to $T$. Hence, by substituting (\ref{06/06/18(11)}),   (\ref{06/06/18(12)}) into  (\ref{06/06/18(10)}) and using the Kolmogorov test, we then obtain the desired result.
\end{proof}

\section{Proof of Theorem \ref{21/05/18(1)} (a)}

We  begin with establishing some useful lemmas.

\begin{lemma}
\label{15/05/2018(1)}
 Let $\gamma\in (0, 1)$.
For any $0<s<t<\infty$, there exists a constant $M_\gamma>0$ such that
\[
\|(-A)^\gamma  e^{tA} - (-A)^\gamma e^{sA}\|\le M_\gamma\Big(\frac{1}{s^\gamma} -\frac{1}{t^\gamma}\Big),\]
\[
\|(-A)^\gamma A e^{tA} - (-A)^\gamma Ae^{sA}\|\le M_\gamma\Big(\frac{1}{s^{\gamma+1}} -\frac{1}{t^{\gamma+1}}\Big).\]
\end{lemma}
\begin{proof}
It is well-known that for any $\gamma\in (0, 1)$ and $k=1,\,2$, there exists a constant $M_{\gamma, k}>0$ such that
\begin{equation}
\label{27/05/18(11)}
\|(-A)^{\gamma}A^k e^{tA}\|\le M_{\gamma, k} /t^{k+\gamma},\hskip 15pt t>0.
\end{equation}
 Therefore, by using this estimate and the following equalities
\begin{equation}
\begin{split}
(-A)^\gamma e^{tA} - (-A)^\gamma e^{sA} &= \int^t_s (-A)^\gamma Ae^{uA}du,\\
(-A)^\gamma Ae^{tA} - (-A)^\gamma Ae^{sA} &= \int^t_s (-A)^\gamma A^2e^{uA}du,
\end{split}
\end{equation}
one can easily obtain   the desired results. The proof is complete now.
\end{proof}

\begin{corollary}
\label{14/06/18(1)}
Let $\gamma\in (0, 1)$ and $\beta>0$. For any $0<s<t<\infty$, there exists a constant $C_{\beta, \gamma}>0$ such that
\begin{equation}
\label{15/05/2018(2)}
\|(-A)^\gamma e^{tA} - (-A)^\gamma e^{sA}\|\le C_{\beta, \gamma} (t - s)^{\beta}\cdot s^{-\beta-\gamma}.
\end{equation}
\end{corollary}
\begin{proof}
First note that for any real numbers $a\ge b\ge 0$ and $0<\delta\le 1$, we have the following inequality
\begin{equation}
\label{16/5/18(1)}
a^\delta - b^\delta\le (a - b)^\delta.
\end{equation}
For $0<\beta\le \gamma$ and any $0<s<t<\infty$, we have by using (\ref{16/5/18(1)}) that
\[
\begin{split}
\frac{1}{s^\gamma} - \frac{1}{t^\gamma} &= \frac{1}{s^\gamma}\frac{t^\gamma - s^\gamma}{t^\gamma}\le  \frac{1}{s^\gamma}\Big(\frac{t^\gamma - s^\gamma}{t^\gamma}\Big)^{\beta/\gamma}\le (t - s)^{\beta}\cdot s^{-\beta-\gamma}.
\end{split}
\]
On the other hand, for $\beta> \gamma$ and $0<s<t<\infty$, we have
\begin{equation}
\label{13/05/19(1)}
\begin{split}
\frac{1}{s^\gamma} - \frac{1}{t^\gamma} =\frac{t^\beta -s^\gamma\cdot t^{\beta-\gamma}}{s^\gamma\cdot t^\beta}\le \frac{t^\beta-s^\gamma\cdot s^{\beta-\gamma}}{s^\gamma\cdot s^\beta} = (t - s)^{\beta}\cdot s^{-\beta-\gamma}.
\end{split}
\end{equation}
Now we have by virtue of Lemma \ref{15/05/2018(1)} the desired inequality.
\end{proof}

 To proceed further, let us denote by $|\cdot|_\infty$ the essential least upper bound norm in $L^\infty([-r, 0], {\mathbb R})$, i.e.,
\begin{equation}
\label{12/06/18(1)}
|a|_\infty :=\hbox{ess}\!\!\sup_{\theta\in [-r, 0]}|a(\theta)|\hskip 15pt \hbox{for any}\hskip 15pt a\in L^\infty([-r, 0], {\mathbb R}).
\end{equation}
For $\gamma \in (0, 1)$, we define
\[
\Gamma(t) = \int^t_0 (-A)^\gamma e^{(t-s)A}a(-s)ds, \hskip 20pt t\in [0, r].\]
\begin{proposition}
\label{31/05/18(1)}
The mapping $\Gamma(\cdot): [0, r]\to {\mathscr L}(X)$, the space of all bounded linear operators on $X$, is uniformly bounded and for any $0<\beta<1-\gamma$, there exists a number $C_{\beta, \gamma}>0$ such that
\begin{equation}
\label{10/06/18(5)}
\|\Gamma(t) -\Gamma(s)\|\le C_{\beta, \gamma} (t-s)^{\beta},\hskip 15pt 0\le s<t\le r,
\end{equation}
i.e., $\Gamma$ is H\"older continuous on $[0, r]$ with order $\beta\in (0, 1-\gamma)$.
\end{proposition}
\begin{proof}
First, it is easy to see by virtue of (\ref{10/06/18(2)}) that
\[
\begin{split}
\|\Gamma(t)\| \le |a|_\infty \int^t_0 \|(-A)^\gamma e^{(t-s)A}\|ds\le |a|_\infty \int^t_0 \frac{M_\gamma}{(t-s)^\gamma}ds= \frac{|a|_\infty M_\gamma}{1-\gamma}\cdot t^{1-\gamma},
\end{split}
\]
which immediately implies
\[
\|\Gamma\|_\infty := \sup_{0\le t\le r}\|\Gamma(t)\| \le \frac{|a|_\infty M_\gamma r^{1-\gamma}}{1-\gamma}<\infty.\]
Thus, $\Gamma$ is uniformly bounded in $[0, r]$.
To show the relation (\ref{10/06/18(5)}), we have for $0\le s<t\le r$ that
\begin{equation}
\label{27/05/18(50)}
\begin{split}
&\|\Gamma(t)-\Gamma(s)\|\\
& =  \Big\|\int^t_0 (-A)^\gamma e^{(t-u)A}a(-u)du - \int^s_0 (-A)^\gamma e^{(s-u)A}a(-u)du\Big\|\\
&\le \Big\|\int^t_s (-A)^\gamma e^{(t-u)A}a(-u)du\Big\| + \Big\|\int^s_0 \Big[(-A)^\gamma e^{(t-u)A}a(-u) - (-A)^\gamma e^{(s-u)A}a(-u)\Big]du\Big\|\\
&=: I_1 + I_2.
\end{split}
\end{equation}
By using the relation (\ref{10/06/18(2)}), we easily obtain the inequality
\begin{equation}
\label{27/05/18(51)}
\begin{split}
I_1 &\le |a|_\infty \int^t_s \|(-A)^\gamma e^{(t-u)A}\|du\\
 &\le |a|_\infty \int^t_s \frac{M_\gamma}{(t-u)^\gamma}du= \frac{|a|_\infty M_\gamma}{1-\gamma}(t-s)^{1-\gamma}.
\end{split}
\end{equation}
On the other hand, we can apply Corollary \ref{14/06/18(1)} to (\ref{27/05/18(50)}) to obtain
\begin{equation}
\label{27/05/18(52)}
\begin{split}
I_2 &\le |a|_\infty \int^s_0 \|(-A)^\gamma e^{(t-u)A} - (-A)^\gamma e^{(s-u)A}\|du\\
&\le |\alpha|_\infty M_\gamma (t-s)^{\beta}\int^s_0 (s-u)^{-\beta-\gamma}du\\
&\le \frac{|a|_\infty M_\gamma (t - s)^{\beta}}{1- \gamma -\beta}r^{1-\gamma -\beta}.
\end{split}
\end{equation}
Hence, substituting (\ref{27/05/18(51)}) and (\ref{27/05/18(52)}) into (\ref{27/05/18(50)}), we immediately obtain the desired result.
\end{proof}

To show Theorem \ref{21/05/18(1)} (a), we intend to develop an induction scheme. We first consider the case that $n=0$ and set
\[
V(t) = (-A)^\gamma(G(t)- e^{tA}),\hskip 20pt t\in [0, r].\]
Then, the integral equation to be satisfied by $V(t)$ is
\begin{equation}
\label{01/06/18(2)}
V(t) = V_0(t) + \int^t_0 \Gamma(t-s)A_2(-A)^{-\gamma} V(s)ds,\hskip 20pt t\in [0, r],
\end{equation}
where
\[
\begin{split}
V_0(t) &= \int^t_0 \Gamma(t-s)A_2e^{sA}ds\\
&= \int^t_0 (\Gamma(t-s)-\Gamma(t))A_2 e^{sA}ds + \Gamma(t) A_2A^{-1}(e^{tA}-I),\hskip 15pt t\in [0, r].
\end{split}
\]
Since ${\mathscr D}(A)\subset {\mathscr D}(A_i)$, we have $A_iA^{-1}\in {\mathscr L}(X)$, $i=1,\,2$. Then for any $0<\beta<1-\gamma$, it follows by virtue of Proposition \ref{31/05/18(1)} and (\ref{10/06/18(1)}) that
\[
\begin{split}
\|V_0(t)\| &\le \int^t_0 \|(\Gamma(t-s)-\Gamma(t))A_2A^{-1} Ae^{sA}\|ds + \|\Gamma(t)\|\|A_2A^{-1}\|(\|e^{tA}\| + 1)\\
&\le \int^t_0 C_{\beta,\gamma}M \|A_2 A^{-1}\|s^{\beta-1}ds + \|\Gamma\|_\infty \|A_2A^{-1}\|(M+1)\\
&\le \frac{C_{\beta,\gamma} M\|A_2 A^{-1}\|}{\beta}r^{\beta} + \|\Gamma\|_\infty \|A_2 A^{-1}\|(M+1),\hskip 20pt t\in [0, r].
\end{split}
\]
Hence, $V_0(\cdot)$ is uniformly bounded on $[0, r]$. Since ${\mathscr D}((-A)^\nu)\subset {\mathscr D}(A_2)$, we have
\begin{equation}
\label{10/05/19(1)}
\begin{split}
\|(-A)^\gamma e^{tA}A_2(-A)^{-\gamma}\| \le M_\gamma\cdot \|A_2(-A)^{-\nu}\|\cdot \|(-A)^{-(\gamma-\nu)}\|\cdot t^{-\gamma},\hskip 15pt t>0.
\end{split}
\end{equation}
Hence, by virtue of the well-known Gronwall lemma and (\ref{01/06/18(2)}) that
\[
v_0 := \sup_{0\le t\le r}\|V(t)\|<\infty.\]
Note that
\begin{equation}
\label{01/06/18(4)}
(-A)^\gamma G(t) = (-A)^\gamma e^{tA} + V(t),\hskip 20pt t\in [0, r].
\end{equation}
Hence, for $t\in (0, r]$ we have
\[
\begin{split}
\|(-A)^\gamma G(t)\| &\le \|(-A)^\gamma e^{tA}\| + \|V(t)\|\\
&\le M_\gamma/t^\gamma + v_0\\
& \le (M_\gamma + v_0 r^\gamma)/t^\gamma,\hskip 15pt M_\gamma>0,
\end{split}
\]
which is (\ref{4567}) with $n=0$. In a similar manner,
for any $0< s<t\le r$, we have from (\ref{10/06/18(2)}) and (\ref{01/06/18(4)}) that
\[
\begin{split}
\Big\|\int^t_s (-A)^\gamma G(u)du\Big\| &\le \int^t_s \|(-A)^\gamma e^{uA}\|du + \int^t_s \|V(u)\|du\\
&\le \int^t_s \frac{M_\gamma}{u^\gamma}du + v_0 (t-s)\le \frac{M_\gamma\cdot r^{1-\gamma}}{1-\gamma} + rv_0 <\infty,
\end{split}
\]
which is (\ref{4569}) with $n=0$.

Now suppose the fundamental solution $G(\cdot)$ satisfies all the estimates in Theorem \ref{21/05/18(1)} on the intervals $[0, r]$, $\cdots$, $[(n-1)r, nr]$. Then in the interval $[nr, (n+1)r]$, the integral equation to be satisfied by
\[
V(t) = (-A)^\gamma\Big(G(t) - \int^t_{nr} e^{(t-s)A}A_1G(s-r)ds\Big),\hskip 20pt t\in [nr, (n+1)r],\]
is
\[
V(t) = V_0(t) + \int^t_{nr} \Gamma(t-u)A_2(-A)^{-\gamma}V(u)du,\]
where
\begin{equation}
\label{07/06/18(10)}
\begin{split}
V_0(t) &= (-A)^\gamma e^{tA} + (-A)^\gamma \int^{nr}_r e^{(t-s)A}A_1G(s-r)ds + \int^{t-r}_0 e^{(t-r-u)A}\Gamma(r)A_2 G(u)du \\
&\,\,\,\,\, + \int^{nr}_{t-r}\Gamma(t-u)A_2G(u)du + \int^t_{nr} \Gamma(t-u)A_2\int^u_{nr}e^{(u-s)A}A_1G(s-r)dsdu\\
&=: I_1(t) +I_2(t) +I_3(t) +I_4(t)+I_5(t).
\end{split}
\end{equation}
Now we estimate each term on the right hand side of (\ref{07/06/18(10)}). First, note that
\[
\|I_1(t)\|\le \frac{M_\gamma}{t^\gamma}\le \frac{M_\gamma}{(nr)^\gamma},\hskip 20pt t\in (nr, (n+1)r],\]
and
\begin{equation}
\label{07/06/18(20)}
I_2(t) =   \sum^{n-1}_{j=1} (-A)^\gamma \int^{(j+1)r}_{jr} e^{(t-s)A}A_1 G(s-r)ds,\hskip 20pt t\in (nr, (n+1)r].
\end{equation}
By virtue of condition (H), for any $\mu\vee\gamma\le \tau<1$, ${\mathscr D}((-A)^\tau)\subset {\mathscr D}(A_i)$ and $A_i(-A)^{-\tau}\in {\mathscr L}(X)$, $i=1,\,2$.
 For $t\in (nr, (n+1)r]$, $j=0,\,1,\,\cdots,\, n-2$, and $\gamma\vee\mu +\delta<1$ with $\delta>0$ sufficiently small, it is easy to see, by the induction assumption and Corollary \ref{15/05/2018(1)}, that
\begin{equation}
\label{07/06/18(21)}
\begin{split}
 &\Big\|(-A)^\gamma \int^{(j+1)r}_{jr} e^{(t-s)A}A_1 G(s-r)ds\Big\|\\
&\le \int^{(j+1)r}_{jr}\Big\|(-A)^\gamma  e^{(t-s)A} - (-A)^\gamma  e^{(t-jr)A}\Big\|\cdot \|A_1(-A)^{-(\gamma\vee\mu +\delta)}\|\cdot \|(-A)^{\gamma\vee\mu +\delta} G(s-r)\|ds\\
&\,\,\,\,\,\, + \|(-A)^\gamma  e^{(t-jr)A}\|\cdot \|A_1(-A)^{-(\gamma\vee\mu +\delta)}\|\cdot \Big\|\int^{(j+1)r}_{jr} (-A)^{\gamma\vee\mu +\delta} G(s-r)ds\Big\|\\
&\le  C_{\gamma, \mu, \delta} \int^{(j+1)r}_{jr}\frac{(s-jr)^{\gamma}}{(t-s)^{2\gamma}}\cdot \frac{C'_{j}}{(s-jr)^{\gamma\vee\mu +\delta}}ds + M_{\gamma, \mu,\delta} C'_{j}\cdot \frac{1}{(t-jr)^{\gamma}}\\
  &\le  \frac{C_{\gamma, \mu, \delta} C'_{j}}{r^{2\gamma}}\int^{(j+1)r}_{jr}\frac{1}{(s-jr)^{\gamma\vee\mu +\delta-\gamma}}ds + \frac{M_{\gamma, \mu, \delta} C'_{j}}{(t-jr)^{\gamma}}\\
  &\le \frac{C_{\gamma, \mu, \delta} C'_{j}}{r^{2\gamma}}\cdot \frac{r^{1-(\gamma\vee\mu +\delta-\gamma)}}{1-(\gamma\vee\mu +\delta-\gamma)} +\frac{M_{\gamma, \mu, \delta}C'_{j}}{r^{\gamma}},\hskip 15pt C_{\gamma, \mu, \delta},\,C'_j,\,M_{\gamma, \mu,\delta}>0.
\end{split}
\end{equation}
On the other hand, for $j=n-1$ and $t\in (nr, (n+1)r]$, we have by using Corollary \ref{15/05/2018(1)}
and induction assumption  that for $0<\beta<1-\gamma$ and $\delta>0$ sufficiently small,
\begin{equation}
\label{07/06/18(12)}
\begin{split}
&\Big\| (-A)^\gamma \int^{nr}_{(n-1)r} e^{(t-s)A}A_1G(s-r)ds\Big\|\\
&\le \int^{nr}_{(n-1)r} \|(-A)^\gamma e^{(t-s)A} - (-A)^\gamma e^{(t-(n-1)r)A}\|\cdot\|A_1(-A)^{-(1-\delta)}\|\cdot \|(-A)^{1-\delta}G(s-r)\|ds\\
&\,\,\,\,+ \|(-A)^\gamma e^{(t-(n-1)r)A}\|\cdot \|A_1(-A)^{-(1-\delta)}\|\cdot \Big\|\int^{nr}_{(n-1)r} (-A)^{1-\delta}G(s-r)ds\Big\|\\
&\le C_{\beta, \gamma, \delta}\int^{nr}_{(n-1)r} \frac{(s-(n-1)r)^{\beta}}{(t-s)^{\beta+\gamma}}
\cdot\frac{C'_{n-1}}{(s-(n-1)r)^{1-\delta}}ds + \frac{M_{\delta,\gamma}}{{(t-(n-1)r)}^\gamma}C'_{n-1}\\
&\le C_{\beta, \gamma, \delta}\int^{nr}_{(n-1)r} \frac{(s-(n-1)r)^{\beta}}{(nr-s)^{\beta+\gamma}}
\cdot\frac{C'_{n-1}}{(s-(n-1)r)^{1-\delta}}ds + \frac{M_{\delta, \gamma}\cdot C'_{n-1}}{r^\gamma}\\
&= C_{\beta, \gamma, \delta}\cdot C'_{n-1}\cdot B(1-\beta-\gamma, \beta+\delta)+ \frac{M_{\delta, \gamma}\cdot C'_{n-1}}{r^\gamma},
\end{split}
\end{equation}
where  $C_{\beta, \gamma, \delta}, \,M_{\delta,\gamma},\, C'_{n-1}>0$ and $B(\cdot, \cdot)$ is the standard Beta function.
Combining (\ref{07/06/18(20)}), (\ref{07/06/18(21)}) and (\ref{07/06/18(12)}), we have for $t\in (nr, (n+1)r]$ that
\begin{equation}
\label{07/06/18(14)}
\begin{split}
\|I_2(t)\|\le &\,\,\sum^{n-2}_{j=1} \Big(\frac{C_{\gamma, \mu,\delta} C'_{j}r^{1-\gamma\vee\mu-\delta+\gamma}}{1-(\gamma\vee\mu+\delta-\gamma)}
+\frac{M_{\gamma, \mu, \delta} C'_{j}}{r^{\gamma}}\Big)\\
 &\,\, +C_{\beta, \gamma, \delta}\cdot C'_{n-1}\cdot B(1-\beta-\gamma, \beta+\delta)+ \frac{M_{\delta, \gamma}\cdot C'_{n-1}}{r^\gamma}.
\end{split}
\end{equation}

To estimate $I_3(t)$, we first note that
\begin{equation}
\label{01/06/18(1)}
\|e^{tA}- e^{sA}\|\le C_\alpha (t-s)^\alpha\cdot s^{-\alpha},\hskip 15pt C_\alpha>0,
\end{equation}
for any $0<s<t<\infty$ and $\alpha\in (0, 1]$. Indeed, by virtue of (\ref{10/06/18(1)}), we have for any $0<s<t$ that
\begin{equation}
\label{10/06/18(10)}
\begin{split}
\|e^{tA}- e^{sA}\|\le \int^t_s \|Ae^{uA}\|du &\le M\int^t_s \frac{1}{u}du\\
& = M\log (t/s)= M\log \Big(1 + \frac{t-s}{s}\Big).
\end{split}
\end{equation}
It is elementary to see that
\begin{equation}
\label{14/06/18(10)}
\log(1+a)\le a^\alpha/\alpha\hskip 15pt \hbox{for any}\hskip 15pt a>0,\,\,\, 0<\alpha\le 1.
\end{equation}
 Hence, (\ref{01/06/18(1)}) follows immediately from (\ref{10/06/18(10)}).
Now we re-write $I_3(t)$ as
\begin{equation}
\label{08/06/18(1)}
\begin{split}
\int^{t-r}_0& e^{(t-u-r)A}\Gamma(r) A_2 G(u)du\\
& = \sum^{n-2}_{j=0} \int^{(j+1)r}_{jr} e^{(t-u-r)A} \Gamma(r) A_2 G(u)du + \int^{t-r}_{(n-1)r} e^{(t-u-r)A} \Gamma(r) A_2 G(u)du.
\end{split}
\end{equation}
For $j=0,\,1,\,\cdots,\, n-2$, we have by the induction assumption and
 (\ref{01/06/18(1)}) that for sufficiently small $\delta>0$,
 \begin{equation}
\label{08/06/18(2)}
\begin{split}
&\Big\|\int^{(j+1)r}_{jr} e^{(t-u-r)A}\Gamma(r) A_2 G(u)du\Big\|\\
&\le \int^{(j+1)r}_{jr}\| e^{(t-u-r)A} -e^{(t-jr-r)A}\|\cdot \|\Gamma(r)\|\cdot \|A_2(-A)^{-(1-\delta)}\|\cdot \|(-A)^{1-\delta} G(u)\|du\\
&\,\,\,\,\, + \|e^{(t-jr-r)A}\|\cdot \|\Gamma(r)\|\cdot \|A_2(-A)^{-\nu}\|\cdot \Big\|\int^{(j+1)r}_{jr}(-A)^{\nu} G(u)du\Big\|\\
&\le \int^{t-r}_{jr} C_{\gamma,\delta}\cdot \frac{(u-jr)^\gamma}{(t-r-u)^\gamma}\cdot \frac{C'_j}{(u-jr)^{1-\delta}}du + M_{\nu}C'_j\\
&\le C_{\gamma, \delta}C'_j\cdot B(1-\gamma, \gamma + \delta) + M_{\nu}C'_j,\,\,\,\,\,C_{\gamma, \delta}>0,\,\,C'_j>0,\, M_\nu>0.
\end{split}
\end{equation}
 In a similar manner, we have for $t\in (nr, (n+1)r]$ and sufficiently small $\delta>0$ that
\begin{equation}
\label{08/06/18(3)}
\begin{split}
\Big\|\int^{t-r}_{(n-1)r} &e^{(t-u-r)A} \Gamma(r) A_2 G(u)du\Big\|\\
 &\le \int^{t-r}_{(n-1)r}\| e^{(t-u-r)A} -e^{(t-nr)A}\|\cdot \|\Gamma(r)\|\cdot\|A_2(-A)^{-(1-\delta)}\|\cdot \|(-A)^{1-\delta} G(u)\|du\\
&\,\,\,\, + \|e^{(t-nr)A}\|\cdot \|\Gamma(r)\|\cdot\|A_2(-A)^{-(1-\delta)}\|\cdot \Big\|\int^{t-r}_{(n-1)r} (-A)^{1-\delta}G(u)du\Big\|\\
&\le C_{\gamma, \delta}\int^{t-r}_{(n-1)r} \frac{(u-(n-1)r)^\gamma}{(t-r-u)^\gamma}\cdot \frac{1}{(u-(n-1)r)^{1-\delta}}du + M_\delta C'_{n-1, \delta}\\
&= C_{\gamma, \delta}B(1-\gamma, \gamma+\delta) + M_\delta C'_{n-1, \delta},\hskip 15pt C_{\gamma, \delta},\, C'_{n-1, \delta},\, M_\delta>0.
\end{split}
\end{equation}
By virtue of (\ref{08/06/18(1)}), (\ref{08/06/18(2)}) and (\ref{08/06/18(3)}), we thus obtain for $t\in (nr, (n+1)r]$ that
\begin{equation}
\label{12/06/18(60)}
\begin{split}
\Big\|&\int^{t-r}_0 e^{(t-u-r)A}\Gamma(r) A_2 G(u)du\Big\| \le (M_\delta + C_{\gamma, \delta}B(1-\gamma, \gamma+\delta))\Big(\sum^{n-1}_{j=0} C'_{j, \delta} +1\Big).
\end{split}
\end{equation}

Now we intend to estimate $I_4(t)$. To this end, we have by virtue of (H), (\ref{10/06/18(5)}) and induction assumption that for sufficiently small $\delta>0$,
\[
\begin{split}
\Big\|\int^{nr}_{t-r} \Gamma(t-u)A_2 G(u)du\Big\|
 &\le \int^{nr}_{t-r} \|\Gamma(t-u)-\Gamma(r)\|\cdot \|A_2(-A)^{-(1-\delta)}\|\cdot\|(-A)^{1-\delta}G(u)\|du\\
  &\,\,\,\,\,+ \|\Gamma(r)\| \cdot\|A_2(-A)^{-(1-\delta)}\|\cdot \Big\|\int^{nr}_{t-r} (-A)^{1-\delta} G(u)du\Big\|\\
&\le \int^{nr}_{t-r} C_{\beta,\gamma, \delta}(u-t+r)^{\beta}\cdot \frac{C'_{n-1}}{(u-t+r)^{1-\delta}}du + \|\Gamma(r)\|C'_{n-1, \delta}\\
&\le \frac{C_{\beta,\gamma, \delta}C'_{n-1}r^{(\beta+\delta)}}{\beta+\delta} + \|\Gamma(r)\|\cdot C'_{n-1, \delta},
\end{split}
\]
where $0<\beta<1-\gamma$ and $C_{\beta, \gamma, \delta},\,\,C'_{n-1, \delta}>0$.

Last, we are in a position to estimate $I_5(t)$. By assumption, we can obtain in terms of Corollary \ref{14/06/18(1)} and induction assumption with $\delta>0$ sufficiently small that for  $nr<u\le (n+1)r$,
\begin{equation}
\label{08/06/18(5)}
\begin{split}
&\Big\|(-A)^\nu \int^u_{nr} e^{(u-s)A} A_1 G(s-r)ds\Big\|\\
 &\le \Big\|\int^u_{nr} (-A)^\nu (e^{(u-s)A} - e^{(u-nr)A})A_1(-A)^{-(1-\delta)}(-A)^{1-\delta} G(s-r)ds\Big\|\\
  &\,\,\,\,\,+ \|(-A)^\nu e^{(u-nr)A}\|\cdot \|A_1(-A)^{-(1-\delta)}\|\cdot\Big\|\int^u_{nr}(-A)^{1-\delta} G(s-r)ds\Big\|\\
 &\le C_{\tilde\beta, \nu, \delta}\int^u_{nr} \frac{(s-nr)^{\tilde\beta}}{(u-s)^{\tilde\beta+\nu}}\cdot \frac{C'_{\tilde\beta, n}}{(s-nr)^{1-\delta}}ds + \frac{C'_{n, \delta}\cdot M_\nu}{(u-nr)^\nu},\hskip 15pt C_{\tilde\beta, \nu, \delta},\,C'_{\tilde\beta, n},\, C'_{n, \delta},\, M_\nu>0,
\end{split}
\end{equation}
where $0<\tilde\beta<1-\nu$. Hence, we have for $t\in (nr, (n+1)r]$ that
\begin{equation}
\label{08/06/18(575690)}
\begin{split}
\|I_5(t)\| &\le \int^t_{nr} \|\Gamma(t-u)\|\cdot \|A_2(-A)^{-\nu}\|\cdot \Big\|(-A)^\nu \int^u_{nr} e^{(u-s)A} A_1 G(s-r)ds\Big\|du\\
&\le \|\Gamma\|_\infty  \|A_2(-A)^{-\nu}\|\int^t_{nr}\int^u_{nr} \frac{C_{\tilde\beta, \nu, \delta}C'_{\tilde\beta, n}(s-nr)^{\tilde\beta+\delta-1}}{(u-s)^{\nu+\tilde\beta}}dsdu + \int^t_{nr}\frac{C'_{n, \delta}M_\nu}{(u-nr)^\nu}du\\
&\le \frac{r^{1-\nu}\|\Gamma\|_\infty  \|A_2(-A)^{-\nu}\|}{1-\nu-\tilde\beta}\cdot \frac{C_{\tilde\beta, \nu, \delta}C'_{\tilde\beta, n}}{\tilde\beta+\delta} + \frac{C'_{n, \delta}\cdot r^{\nu}M_\nu}{1-\nu}.
\end{split}
\end{equation}
By combining (\ref{07/06/18(14)}), (\ref{12/06/18(60)}) with (\ref{08/06/18(575690)}) and noticing (\ref{10/05/19(1)}),  we obtain that $V_0(t)$ and $V(t)$ both are uniformly bounded in $[nr, (n+1)r]$.

Finally, for $nr<t\le (n+1)r$,
we have by (H) and induction assumption that for sufficiently small $\delta>0$,
\begin{equation}
\label{08/06/18(567036767)}
\begin{split}
&\Big\|(-A)^\gamma \int^t_{nr} e^{(t-s)A} A_1 G(s-r)ds\Big\|\\
&\le  \Big\|\int^t_{nr} (-A)^\gamma(e^{(t-s)A}-e^{(t-nr)A})A_1(-A)^{-(1-\delta)}(-A)^{1-\delta}G(s-r)ds\Big\|\\
&\,\,\,\,\, + \|(-A)^\gamma e^{(t-nr)A}\|\cdot \Big\|\int^t_{nr} A_1(-A)^{-(1-\delta)}(-A)^{1-\delta}G(s-r)ds\Big\|\\
 &\le C'_{\beta, \gamma, n}\|A_1(-A)^{-(1-\delta)}\|\int^t_{nr} (t-s)^{1-\beta-\gamma-1}\cdot (s-nr)^{\beta-1+\delta}ds + \frac{C'_n\|A_1(-A)^{-(1-\delta)}\| M_{\gamma, \delta}}{(t-nr)^\gamma}\\
 &\le \{C''_{\beta, \gamma, n}\cdot B(1-\beta-\gamma, \beta+\delta) + C'_n\|A_1(-A)^{-(1-\delta)}\| M_{\gamma, \delta}\}/(t-nr)^\gamma,
\end{split}
\end{equation}
where $C'_{\beta, \gamma, n},\,C''_{\beta, \gamma, n}, C'_{n, \delta},\, M_{\gamma, \delta}>0$, $0<\beta<1-\gamma$ and $B(\cdot, \cdot)$ is the Beta function.

 Now the inequality (\ref{4567}) for $t\in (nr, (n+1)r]$ follows from  (\ref{08/06/18(567036767)}) and
the equality
\begin{equation}
\label{08/06/18(5756906)}
(-A)^\gamma G(t) = (-A)^\gamma \int^t_{nr} e^{(t-s)A}A_1 G(s-r)ds + V(t).
\end{equation}
On the other hand, by using (\ref{08/06/18(567036767)}), we have that
\[
\begin{split}
\Big\|\int^t_{nr}(-A)^\gamma G(u)du\Big\| & \le \int^t_{nr} \Big\|(-A)^\gamma \int^u_{nr} e^{(u-s)A}A_1G(s-r)ds\Big\|du + \int^t_{nr}\|V(u)\|du\\
&\le C'''_{\gamma, n}\Big(\frac{r^{1-\gamma}}{1-\gamma} +1\Big),\,\,\,\,\,C'''_{\gamma, n}>0,
\end{split}
\]
is uniformly bounded in $[nr, (n+1)r]$, a fact which implies (\ref{4569}). The proof is complete now.

\section{Proof of Theorem 2.1 (b)}

We still want to develop an induction scheme here. First, consider $V(t)$ and $V_0(t)$, $t\in [0, r]$, defined in (\ref{01/06/18(2)}).
\begin{lemma}
\label{01/06/18(10)}
For any $0< s<t< r$ and $0<\beta<1-\gamma$, there exists a constant $M_{\beta, \gamma}>0$ such that
\begin{equation}
\label{01/06/18(3)}
\|V(t)-V(s)\|\le M_{\beta, \gamma}(t-s)^{\beta}s^{-\beta}.
\end{equation}
\end{lemma}
\begin{proof}  We notice, by definition, for $0< s<t< r$ that
\begin{equation}
\label{11/06/18(1)}
\begin{split}
V_0(t)-V_0(s) = &\,\, \int^t_s \Gamma(t-u)A_2 e^{uA}du + \int^s_0 (\Gamma(t-u)A_2 e^{uA} -\Gamma(s-u)A_2 e^{uA})du\\
=:&\,\, I_1(s, t) + I_2(s, t).
\end{split}
\end{equation}
For $0< s< t< r$, it is easy to see from (\ref{14/06/18(10)}) that
\begin{equation}
\label{11/06/18(2)}
\begin{split}
\|I_1(s, t)\| &\le \int^t_s \|\Gamma(t-u)\|\cdot \|A_2 A^{-1}\|\cdot \|A e^{uA}\|du\\
&\le  M\|\Gamma\|_\infty\cdot \|A_2 A^{-1}\| \ln\Big(1+\frac{t-s}{s}\Big)\le \frac{M\|\Gamma\|_\infty\cdot \|A_2 A^{-1}\|}{\beta}(t-s)^{\beta}s^{-\beta},
\end{split}
\end{equation}
where $M>0$ and by Proposition \ref{31/05/18(1)}, we have for $0< s< t< r$ that
\begin{equation}
\label{11/06/18(3)}
\begin{split}
\|I_2(s, t)\| &\le \int^s_0 \|\Gamma(t-u)-\Gamma(s-u)\|\cdot\|A_2(-A)^{-\nu}\|\cdot \|(-A)^\nu e^{uA}\|du\\
&\le C_{\beta, \gamma}\int^s_0 (t-s)^{\beta} \|A_2(-A)^{-\nu}\|\cdot \frac{M_\nu}{u^\nu}du\\
&\le C_{\beta, \gamma}M_\nu  \frac{\|A_2(-A)^{-\nu}\|}{1-\nu}r^{1-\nu+\beta}(t-s)^{\beta}s^{-\beta},\hskip 15pt C_{\beta, \gamma}>0,\,\,\,\, M_\nu>0.
\end{split}
\end{equation}
Hence, by combining (\ref{11/06/18(1)}), (\ref{11/06/18(2)}) and (\ref{11/06/18(3)}),  we get for $0< s<t< r$ that
\[
\|V_0(t)- V_0(s)\| \le M_{\beta, \gamma}(t-s)^{\beta} s^{-\beta},\hskip 15pt M_{\beta, \gamma}>0,\]
from which the desired result follows by (\ref{10/05/19(1)}) and the well-known Gronwall inequality.
\end{proof}

By virtue of Corollary \ref{14/06/18(1)} and (\ref{01/06/18(3)}), we obtain the estimates in Theorem \ref{21/05/18(1)}
 for $0<s<t< r$:
\[
\begin{split}
\|(-A)^\gamma (G(t) - G(s))\| &\le \|(-A)^\gamma e^{tA} - (-A)^\gamma e^{sA}\| + \|V(t)-V(s)\|\\
&\le M_{\beta, \gamma} \big[(t-s)^{\beta} \cdot s^{-\beta-\gamma} + (t-s)^{\beta} s^{-\beta}\big]\\
&\le (M_{\beta, \gamma} + M_{\beta, \gamma}r^\gamma)(t-s)^{\beta}\cdot s^{-\beta-\gamma},\hskip 10pt M_{\beta, \gamma}>0,\hskip 10pt 0<\beta<1-\gamma.
\end{split}
\]

Now suppose that $G(t)$ satisfies those estimates in Theorem \ref{21/05/18(1)} on the intervals $[0, r]$, $\cdots$, $[(n-1)r, nr]$. Then in the interval $[nr, (n+1)r]$, the integral equation to be satisfied by
\begin{equation}
\label{14/06/18(70)}
V(t) = (-A)^\gamma\Big(G(t) - \int^t_{nr} e^{(t-s)A}A_1G(s-r)ds\Big),\hskip 20pt t\in (nr, (n+1)r],
\end{equation}
is
\[
V(t) = V_0(t) + \int^t_{nr} \Gamma(t-u)A_2(-A)^{-\gamma}V(u)du,\]
where $V_0(t)$ is given as in (\ref{07/06/18(10)}).
We first show the H\"older continuity of $V_0(t)$ and $V(t)$. Let $nr<s<t<(n+1)r$. By virtue of (\ref{07/06/18(10)}), we have
\begin{equation}
\label{07/06/18(10678)}
\begin{split}
V_0(t) -V_0(s)
 =&\,\, (-A)^\gamma (e^{tA}- e^{sA})\\
&\,\, + (-A)^\gamma \Big(\int^{nr}_r e^{(t-u)A}A_1G(u-r)du - \int^{nr}_r e^{(s-u)A}A_1G(u-r)du\Big)\\
&\,\, + \Big(\int^{t-r}_0 e^{(t-r-u)A}\Gamma(r)A_2 G(u)du - \int^{s-r}_0 e^{(s-r-u)A}\Gamma(r)A_2 G(u)du\Big)\\
&\,\, + \Big(\int^{nr}_{t-r}\Gamma(t-u)A_2G(u)du -\int^{nr}_{s-r}\Gamma(s-u)A_2G(u)du\Big)\\
&\,\, + \Big(\int^t_{nr} \Gamma(t-u)A_2\int^u_{nr}e^{(u-v)A}A_1G(v-r)dv du\\
&\,\, - \int^s_{nr} \Gamma(s-u)A_2\int^u_{nr}e^{(u-v)A}A_1G(v-r)dv du\Big)\\
 =:&\,\, I_1(t, s) + I_2(t, s) +I_3(t, s)+I_4(t, s) +I_5(t, s).
\end{split}
\end{equation}
First, by virtue of Corollary \ref{15/05/2018(1)} we have
\begin{equation}
\label{13/06/18(0)}
\begin{split}
\|I_1(t, s)\|&\le M_\gamma (t-s)^{\beta} s^{-\beta-\gamma}\\
&\le \frac{M_\gamma}{(nr)^\gamma}(t-s)^{\beta}\cdot (s-nr)^{-\beta},\hskip 15pt M_\gamma>0.
\end{split}
\end{equation}
Now let us consider the item $I_2(t, s)$. For sufficiently small $\delta>0$, there exists, by virtue of (\ref{07/06/18(12)}) and (\ref{01/06/18(1)}), a value $c_{n, \beta, \gamma, \delta}>0$ such that
\begin{equation}
\label{13/06/18(1)}
\begin{split}
\|I_2(s, t)\| &\le \|e^{(t-nr)A}-e^{(s-nr)A}\|\cdot\sum^n_{k=2}\Big\|(-A)^\gamma \int^{kr}_{(k-1)r} e^{(nr-u)A}A_1G(u-r)du\Big\|\\
&\le c_{n, \beta, \gamma, \delta} (t-s)^{\beta}\cdot (s-nr)^{-\beta}.
\end{split}
\end{equation}
 As for the third term on the right side of (\ref{07/06/18(10678)}), we have
\begin{equation}
\label{13/06/18(2)}
\begin{split}
I_3(t, s) =&\,\, \int^{t-r}_{s-r} e^{(t-r-u)A}\Gamma(r) A_2 G(u)du\\
&\,\, + \int^{s-r}_{(n-1)r} \Big(e^{(t-r-u)A}- e^{(s-r-u)A} - e^{(t-nr)A} + e^{(s-nr)A}\Big)\Gamma(r)A_2G(u)du\\
&\,\, + \Big(e^{(t-nr)A} -e^{(s-nr)A}\Big)\Gamma(r)\int^{s-r}_{(n-1)r}A_2G(u)du\\
&\,\, + \Big(e^{(t-nr)A} -e^{(s-nr)A}\Big)\int^{(n-1)r}_{0}e^{((n-1)r-u)A}\Gamma(r) A_2G(u)du\\
=: &\,\, J_1(t, s) +J_2(t, s) +J_3(t, s)+J_4(t, s).
\end{split}
\end{equation}
By virtue of (\ref{10/06/18(2)})
 and induction assumption, we have for sufficiently small $\delta>0$ that
\begin{equation}
\label{05/04/19(1)}
\begin{split}
\|J_1(t, s)\|
 &\le \|(-A)^{-\delta}\|\int^{t-r}_{s-r} \|(-A)^\delta e^{(t-r-u)A}\|\cdot \|\Gamma(r)\|\cdot \|A_2(-A)^{-\nu}\|\cdot \|(-A)^\nu G(u)\|du\\
&\le M_{\gamma, \nu, \delta}\int^{t-r}_{s-r} (t-r-u)^{-\delta}\frac{C_{n, \nu}}{(u-(n-1)r)^\nu}du\\
&\le C_{n, \gamma, \delta,\nu}(t-s)^{\beta}\int^{t-r}_{(n-1)r} (t-r-u)^{-\beta-\delta}(u-(n-1)r)^{-\nu}du\\
&\le C_{n,\gamma,\delta, \nu}\cdot
B(1-\beta-\delta, 1-\nu)(t-nr)^{1-\beta-\delta-\nu}(t-s)^{\beta}\\
&\le M'_{n, \delta, \nu, \gamma}(t-s)^{\beta}\cdot (s-nr)^{-\beta},\hskip 15pt M'_{n, \delta, \nu, \gamma}>0.
\end{split}
\end{equation}
On the other hand, by virtue of (\ref{01/06/18(1)}) we have for $(n-1)r<u<s-r$ and sufficiently small $\delta>0$ that
\begin{equation}
\label{13/06/18(3)}
\begin{split}
\|e^{(t-r-u)A}&\, - e^{(s-r-u)A} - e^{(t-nr)A} + e^{(s-nr)A}\|\\
& \le \|e^{(t-r-u)A} - e^{(s-r-u)A}\| + \|e^{(t-nr)A} - e^{(s-nr)A}\|\\
&\le c'_{\beta, \gamma, \delta}\frac{(t-s)^{\beta+\delta}}{(s-r-u)^{\beta+\delta}} + c'_{\beta, \gamma, \delta} \frac{(t-s)^{\beta+\delta}}{(s-nr)^{\beta+\delta}}\\
&\le 2c'_{\beta, \gamma, \delta}\frac{(t-s)^{\beta+\delta}}{(s-r-u)^{\beta+\delta}},\hskip 20pt c'_{\beta, \gamma, \delta}>0.
\end{split}
\end{equation}
In a similar way, we have by virtue of (\ref{01/06/18(1)}) that
\begin{equation}
\label{13/06/18(4)}
\begin{split}
\|&e^{(t-r-u)A} - e^{(s-r-u)A} - e^{(t-nr)A} + e^{(s-nr)A}\|\\
& \le \|e^{(t-r-u)A} - e^{(t-nr)A}\| + \|e^{(s-r-u)A} - e^{(s-nr)A}\|\\
&\le c''_{\alpha, \gamma, \delta}\frac{(u-(n-1)r)^{\beta +\delta}}{(t-r-u)^{\beta+\delta}} + c''_{\alpha, \gamma, \delta} \frac{(u-(n-1)r)^{\beta+\delta}}{(s-r-u)^{\beta+\delta}}\le 2c''_{\beta, \gamma, \delta}\frac{(u-(n-1)r)^{\beta+\delta}}{(s-r-u)^{\beta+\delta}},
\end{split}
\end{equation}
where $c''_{\beta, \gamma, \delta}>0$.
Combining these two inequalities, it follows that
\begin{equation}
\label{13/06/18(5)}
\begin{split}
\|e^{(t-r-u)A}& - e^{(s-r-u)A} - e^{(t-nr)A} + e^{(s-nr)A}\|\\
&\le C_{\beta, \delta} (t-s)^{\beta}(s-r-u)^{-(\beta+\delta)}(u-(n-1)r)^{\delta},\hskip 15pt C_{\beta, \delta}>0,
\end{split}
\end{equation}
for $0<\beta<1-\gamma$ and sufficiently small $\delta>0$. With the aid of (\ref{13/06/18(5)}), we  have for sufficiently small $\delta>0$ that
\begin{equation}
\label{13/06/18(6)}
\begin{split}
\|J_2(t, s)\|
 &\le C_{\beta, \delta} \int^{s-r}_{(n-1)r} (t-s)^{\beta}(s-r-u)^{-(\beta+\delta)}(u-(n-1)r)^{\delta}\\
&\,\,\,\,\,\cdot \|A_2(-A)^{-\nu}\|\cdot\|\Gamma(r)\|\frac{C_{n-1, \delta}}{(u-(n-1)r)^{\nu}}du\\
&= C'_{\beta, \delta} \|\Gamma(r)\|(t-s)^{\beta}\int^{s-r}_{(n-1)r} (s-r-u)^{-(\beta+\delta)}(u-(n-1)r)^{\delta-\nu}du\\
&= c'''_{\beta, \nu, \delta, n}B(1-\beta-\delta, 1+\delta-\nu)(t-s)^{\beta}(s-nr)^{-\beta},\hskip 20pt c'''_{\beta, \gamma, \delta, n}>0.
\end{split}
\end{equation}
In a similar manner, one can have by virtue of (\ref{10/06/18(10)}) and (\ref{14/06/18(10)})
 that
\begin{equation}
\label{13/06/18(7)}
\begin{split}
\|J_3(t, s)\| &\le \|e^{(t-nr)A}-e^{(s-nr)A}\|\cdot \|\Gamma(r)\|\cdot \|A_2(-A)^{-\nu}\|\cdot \Big\|\int^{s-r}_{(n-1)r}(-A)^\nu G(u)du\Big\|\\
&\le c''''_{\beta, \nu, n}\|\Gamma(r)\|\Big(\frac{t-s}{s-nr}\Big)^{\beta},\hskip 15pt c''''_{\alpha,\nu, n}>0,
\end{split}
\end{equation}
and
\begin{equation}
\label{13/06/18(8)}
\begin{split}
\|J_4(t, s)\|
&\le \|e^{(t-nr)A}-e^{(s-nr)A}\|\cdot \|\Gamma(r)\|\cdot \|A_2(-A)^{-\nu}\|\sum^{(n-2)r}_{j=0}\int^{(j+1)r}_{jr}\|(-A)^\nu G(u)\|du\\
&\le c'''''_{\beta, \nu, n} \Big(\frac{t-s}{s-nr}\Big)^{\beta}, \hskip 20pt c'''''_{\beta, \nu, n}>0.
\end{split}
\end{equation}
Combining (\ref{13/06/18(2)})-(\ref{13/06/18(8)}), we conclude that for some $M''_{\alpha, \gamma}>0$,
\[
\begin{split}
\Big\|\int^{t-r}_0 e^{(t-r-u)A}&\Gamma(r)A_2G(u)du - \int^{s-r}_0 e^{(s-r-u)A}\Gamma(r)A_2G(u)du\Big\|\\
&\le {M}''_{\beta, \gamma} (t-s)^{\beta} (s-nr)^{-\beta}.
\end{split}
\]
For the item $I_4(t, s)$, let $(n-1)r<s-r<t-r<u<nr$ and by Proposition \ref{31/05/18(1)}, (\ref{14/06/18(10)}), (\ref{05/04/19(1)}) and induction assumption, we may obtain for sufficiently small $\delta>0$ that
\begin{equation}
\label{13/06/18(999)}
\begin{split}
\Big\|&\int^{nr}_{t-r} \Gamma(t-u)A_2G(u)du - \int^{nr}_{s-r} \Gamma(s-u)A_2G(u)du\Big\|\\
 &= \Big\|\int^{nr}_{t-r} (\Gamma(t-u) - \Gamma(s-u))(A_2(-A)^{-\nu})(-A)^\nu G(u)du\\
  &\,\,\,\,\,- \int^{t-r}_{s-r} \Gamma(s-u)(A_2(-A)^{-(1-\delta)})(-A)^{1-\delta} G(u)du\Big\|\\
 &\le C_{n, \beta, \gamma, \nu}\int^{nr}_{t-r} (t-s)^{\beta}\cdot \frac{1}{(u-(n-1)r)^\nu}du\\
  &\,\,\,\,\,+ \|\Gamma\|_\infty C_{n-1, \nu, \delta} \int^{t-r}_{s-r}\frac{1}{(u-(n-1)r)^{1-\delta}}du\\
 &\le C_{n, \beta, \gamma, \nu, \delta}(t-s)^{\beta}\int^{nr}_{t-r} \frac{1}{(u-(n-1)r)^\nu}du\\
  &\,\,\,\,\,+ \|\Gamma\|_\infty C_{n-1, \nu, \delta} (t-s)^{\beta} \int^{t-r}_{(n-1)r} (t-r-u)^{-\beta}\frac{1}{(u-(n-1)r)^{1-\delta}}du\\
 &\le C_{n, \beta, \gamma, \nu, \delta} (t-s)^{\beta}\frac{r^{1-\nu}}{1-\nu} + C'_{n-1, \gamma, \delta}B(1-\beta, \delta)(t-s)^{\beta} (t-nr)^{-\beta+\delta}\\
 &\le C''_{n, \beta, \gamma, \delta}r^\delta (t-s)^{\beta} (s-nr)^{-\beta},\hskip 15pt C''_{n, \beta, \gamma}>0.
\end{split}
\end{equation}
In a similar way, we can have
\begin{equation}
\label{13/06/18(9)}
\begin{split}
\Big\|\int^t_{nr} &\Gamma(t-u)A_2\int^u_{nr}e^{(u-v)A}A_1G(v-r)dv du - \int^s_{nr} \Gamma(s-u)A_2\int^u_{nr}e^{(u-v)A}A_1G(v-r)dvdu\Big\|\\
&\le {C}'''_{n, \beta, \gamma} (t-s)^{\beta}(s-nr)^{-\beta},\hskip 15pt {C}'''_{n, \beta, \gamma}>0.
\end{split}
\end{equation}
Combining (\ref{07/06/18(10678)})-(\ref{13/06/18(9)}), we conclude that
\[
\|V_0(t) - V_0(s)\|\le M'_{n, \beta, \gamma}(t-s)^{\beta}(s-nr)^{-\beta},\hskip 15pt M'_{n, \beta, \gamma}>0,\]
and further we have
\begin{equation}
\label{13/06/18(978904)}
\|V(t) - V(s)\|\le M''_{n, \beta, \gamma}(t-s)^{\beta}(s-nr)^{-\beta},\hskip 15pt M''_{n, \beta, \gamma}>0,
\end{equation}
for $nr<s<t< (n+1)r$.

On the other hand,  we have by assumption that
\begin{equation}
\label{14/06/18(50)}
\begin{split}
(-A)^\gamma \int^t_{nr} & e^{(t-u)A} A_1G(u-r)du - (-A)^\gamma \int^s_{nr} e^{(s-u)A} A_1 G(u-r)du\\
&= \int^t_s (-A)^\gamma e^{(t-u)A}(A_1(-A)^{-(1-\delta)})(-A)^{1-\delta} G(u-r)du\\
&\,\,\,\, + \int^s_{nr} (-A)^\gamma(e^{(t-u)A}- e^{(s-u)A})(A_1(-A)^{-(1-\delta)})(-A)^{1-\delta} G(u-r)du\\
&=: K_1(t, s) + K_2(t, s),
\end{split}
\end{equation}
for any $nr<s<t< (n+1)r$ and sufficiently small $\delta>0$. Hence, we have by the induction hypothesis and Theorem
\ref{21/05/18(1)} that
\begin{equation}
\label{14/06/18(52)}
\begin{split}
\|K_1(t, s)\| &\le \int^t_s \|(-A)^\gamma e^{(t-u)A}\|\cdot \|A_1(-A)^{-(1-\delta)}\|\cdot \|(-A)^{1-\delta} G(u-r)\|du\\
&\le M(t-s)^\beta\int^t_{nr} (t-u)^{-\beta-\gamma}\|A_1 (-A)^{-(1-\delta)}\|\frac{C_{n, \nu}}{(u-nr)^{1-\delta}}du\\
&\le MC_{n, \nu}\|A_1(-A)^{-(1-\delta)}\|\cdot
B(1-\beta-\gamma, \delta)(t-nr)^{1-\beta-\gamma-1+\delta}(t-s)^{\beta}\\
&\le M'_{n, \beta, \nu, \gamma, \delta}(t-s)^{\beta}\cdot (s-nr)^{-\beta-\gamma},\hskip 15pt M'_{n, \beta, \nu, \gamma, \delta}>0.
\end{split}
\end{equation}
In a similar manner, we have
\begin{equation}
\label{14/06/18(53)}
\begin{split}
\|K_2(t, s)\| &\le \int^s_{nr} M\|A_1(-A)^{-(1-\delta)}\|\cdot (t-s)^{\beta}\cdot (s-nr)^{-\beta-\gamma} (u-nr)^{-(1-\delta)}du\\
&\le \frac{MC_{n, \nu}\|A_1(-A)^{-(1-\delta)}\|(t-s)^{\beta}}{\delta}(s-nr)^{-\beta-\gamma}{(s-nr)^{\delta}}\\
&\le M''_{n, \delta, \beta, \gamma}(t-s)^{\beta}(s-nr)^{-\beta-\gamma},\hskip 15pt M''_{n,\delta, \beta, \gamma}>0.
\end{split}
\end{equation}
With the aid of (\ref{14/06/18(50)})-(\ref{14/06/18(53)}), we thus obtain
 \begin{equation}
\label{14/06/18(55)}
\begin{split}
\Big\|(-A)^\gamma \int^t_{nr}&  e^{(t-u)A} A_1G(u-r)du - (-A)^\gamma \int^s_{nr} e^{(s-u)A} A_1 G(u-r)du\Big\|\\
& \le C_{n, \beta, \gamma}\cdot (t-s)^{\beta}(s-nr)^{-\beta-\gamma}
\end{split}
\end{equation}
for some $C_{n, \beta, \gamma}>0$. Now combining (\ref{14/06/18(70)}), (\ref{13/06/18(978904)}) and (\ref{14/06/18(55)}), we finally obtain (\ref{14/06/18(80)}).
 The proof is thus complete.

\end{document}